\documentclass[a4paper,10pt]{scrartcl}
\usepackage[utf8]{inputenc}
\usepackage[english]{babel}
\usepackage[fixlanguage]{babelbib}
\usepackage{cite}
\usepackage{amssymb, amsmath, amstext, amsopn, amsthm, amscd, amsxtra, amsfonts, commath}
\usepackage{amssymb,amsfonts,amsmath,bbm,mathrsfs,stmaryrd,mathtools}
\usepackage{yfonts, mathrsfs}
\usepackage{colonequals}

\usepackage[shortlabels]{enumitem}
\usepackage{tensor}

\usepackage{graphicx}
\usepackage[all]{xy}
\usepackage{todonotes}
\usepackage{tikz-cd}
\usepackage
{hyperref}%
\hypersetup{
  hidelinks = true,
  urlcolor = red,
  colorlinks = false,
  linkcolor = blue,
  citecolor = blue,
  linktocpage = true,
  pdftitle = {Singularities of the exponential function of a semidirect product},
  pdfauthor = {Alexandru Chirvasitu, Rafael Dahmen, Karl-Hermann Neeb, Alexander Schmeding},
  bookmarksopen = true,
  bookmarksopenlevel = 1,
  unicode = true,
  hypertexnames =false,  
}%
\usepackage{cleveref}

\creflabelformat{enumi}{#2#1#3}

\crefname{section}{Section}{Sections}
\crefformat{section}{#2Section~#1#3} 
\Crefformat{section}{#2Section~#1#3} 

\crefname{subsection}{\S}{\S\S}
\AtBeginDocument{%
  \crefformat{subsection}{#2\S#1#3}%
  \Crefformat{subsection}{#2\S#1#3}%
}

\crefname{subsubsection}{\S}{\S\S}
\AtBeginDocument{%
  \crefformat{subsubsection}{#2\S#1#3}%
  \Crefformat{subsubsection}{#2\S#1#3}%
}

\crefname{defn}{definition}{definitions}
\crefformat{defn}{#2definition~#1#3} 
\Crefformat{defn}{#2Definition~#1#3} 

\crefname{ex}{example}{examples}
\crefformat{ex}{#2example~#1#3} 
\Crefformat{ex}{#2Example~#1#3} 

\crefname{exs}{example}{examples}
\crefformat{exs}{#2example~#1#3} 
\Crefformat{exs}{#2Example~#1#3} 

\crefname{rem}{remark}{remarks}
\crefformat{rem}{#2remark~#1#3} 
\Crefformat{rem}{#2Remark~#1#3} 

\crefname{rems}{remark}{remarks}
\crefformat{rems}{#2remark~#1#3} 
\Crefformat{rems}{#2Remark~#1#3} 

\crefname{convention}{convention}{conventions}
\crefformat{convention}{#2convention~#1#3} 
\Crefformat{convention}{#2Convention~#1#3} 

\crefname{notation}{notation}{notations}
\crefformat{notation}{#2notation~#1#3} 
\Crefformat{notation}{#2Notation~#1#3} 

\crefname{table}{table}{tables}
\crefformat{table}{#2table~#1#3} 
\Crefformat{table}{#2Table~#1#3}

\crefname{lem}{lemma}{lemmas}
\crefformat{lem}{#2lemma~#1#3} 
\Crefformat{lem}{#2Lemma~#1#3} 

\crefname{prop}{proposition}{propositions}
\crefformat{prop}{#2proposition~#1#3} 
\Crefformat{prop}{#2Proposition~#1#3} 

\crefname{propositionN}{proposition}{propositions}
\crefformat{propositionN}{#2proposition~#1#3} 
\Crefformat{propositionN}{#2Proposition~#1#3} 

\crefname{cor}{corollary}{corollaries}
\crefformat{cor}{#2corollary~#1#3} 
\Crefformat{cor}{#2Corollary~#1#3} 

\crefname{corollaryN}{corollary}{corollaries}
\crefformat{corollaryN}{#2corollary~#1#3} 
\Crefformat{corollaryN}{#2Corollary~#1#3} 

\crefname{thm}{theorem}{theorems}
\crefformat{thm}{#2theorem~#1#3} 
\Crefformat{thm}{#2Theorem~#1#3} 

\crefname{thmN}{theorem}{theorems}
\crefformat{thmN}{#2theorem~#1#3} 
\Crefformat{thmN}{#2Theorem~#1#3} 

\crefname{conj}{conjecture}{conjectures}
\crefformat{conj}{#2conjecture~#1#3} 
\Crefformat{conj}{#2Conjecture~#1#3} 

\crefname{conjN}{conjecture}{conjectures}
\crefformat{conjN}{#2conjecture~#1#3} 
\Crefformat{conjN}{#2Conjecture~#1#3} 

\crefname{enumi}{}{}
\crefformat{enumi}{#2#1#3}
\Crefformat{enumi}{#2#1#3}

\crefname{assumption}{assumption}{Assumptions}
\crefformat{assumption}{#2assumption~#1#3} 
\Crefformat{assumption}{#2Assumption~#1#3} 

\crefname{construction}{construction}{Constructions}
\crefformat{construction}{#2construction~#1#3} 
\Crefformat{construction}{#2Construction~#1#3} 

\crefname{question}{question}{Questions}
\crefformat{question}{#2question~#1#3} 
\Crefformat{question}{#2Question~#1#3} 

\crefname{equation}{}{}
\crefformat{equation}{(#2#1#3)} 
\Crefformat{equation}{(#2#1#3)}

\crefname{setup}{setup}{setup}
\crefformat{setup}{#2setup~#1#3} 
\Crefformat{setup}{#2Setup~#1#3}

\numberwithin{equation}{section}

\swapnumbers
\newtheoremstyle{standard}
{16pt} 
{16pt} 
{} 
{} 
{\bfseries}
{} 
{ } 
{{\thmname{#1~}}{\thmnumber{#2.}}\thmnote{~(#3)}} 
\newtheoremstyle{kursiv}
{16pt} 
{16pt} 
{\itshape} 
{} 
{\bfseries}
{} 
{ } 
{{\thmname{#1~}}{\thmnumber{#2.}}\thmnote{~(#3)}} 
\theoremstyle{standard}
\newtheorem{defn}{Definition}[section]
\newtheorem{ex}[defn]{Example}

\newtheorem{rem}[defn]{Remark}
\newtheorem{rems}[defn]{Remarks}

\newtheorem{setup}[defn]{}

\theoremstyle{kursiv}
\newtheorem{thm}[defn]{Theorem}
\newtheorem{prop}[defn]{Proposition}
\newtheorem{cor}[defn]{Corollary}
\newtheorem{lem}[defn]{Lemma}

\newtheorem{thmN}{Theorem}

\newtheorem{conjN}{Conjecture}

\newcommand\bC{{\mathbb C}}

\newcommand\bG{G}

\newcommand\bQ{{\mathbb Q}}
\newcommand\bR{{\mathbb R}}
\newcommand\bS{{\mathbb S}}
\newcommand\bT{{\mathbb T}}

\newcommand\bZ{{\mathbb Z}}

\newcommand{\cA}{\ensuremath{\mathcal{A}}}

\newcommand{\cL}{\ensuremath{\mathcal{L}}}

\newcommand{\Lf}{\ensuremath{\mathbf{L}}}
\newcommand{\dd}{\mathop{}\!\mathrm{d}}

\newcommand{\Q}{\ensuremath{\mathbb{Q}}}

\newcommand{\R}{\ensuremath{\mathbb{R}}}

\newcommand{\N}{\ensuremath{\mathbb{N}}}
\newcommand{\Z}{\ensuremath{\mathbb{Z}}}

\newcommand{\SSS}{\ensuremath{\mathbb{S}}}
\newcommand{\T}{\ensuremath{\mathbb{T}}}

\newcommand{\be}{\ensuremath{\mathbf{e}}}
\newcommand{\bk}{\ensuremath{\mathbf{k}}}
\newcommand{\bm}{\ensuremath{\mathbf{m}}}

\newcommand{\bz}{\ensuremath{\mathbf{z}}}
\newcommand{\bth}{\ensuremath{\mathbf{\theta}}}

\newcommand{\per}{\ensuremath{\mathrm{per}}}

\newcommand{\coloneq}{\colonequals}
\newcommand{\pmat}[1]{\begin{pmatrix} #1 \end{pmatrix}}

\DeclareMathOperator{\one}{\mathbf{1}}
\DeclareMathOperator{\supp}{supp}

\DeclareMathOperator{\Fl}{Fl}
\DeclareMathOperator{\id}{id}
\DeclareMathOperator{\GL}{\mathrm{GL}}

\DeclareMathOperator{\End}{\mathrm{End}}
\DeclareMathOperator{\SO}{\mathrm{SO}}
\DeclareMathOperator{\OO}{\mathrm{O}}

\DeclareMathOperator{\im}{im}
\DeclareMathOperator{\Aut}{Aut}

\DeclareMathOperator{\ad}{ad}
\newcommand{\Frechet}{Fr\'{e}chet}

\newcommand{\Spec}{\mathrm{Spec}}
\newcommand{\Sp}{\mathrm{Sp}}

\newcommand{\xrightarrowdbl}[2][]{%
  \xrightarrow[#1]{#2}\mathrel{\mkern-14mu}\rightarrow
}

\DeclareMathOperator{\BMS}{BMS}

\newcommand{\func}[5]{#1 \colon #2 \rightarrow #3 ,\quad #4 \mapsto #5}

\newcommand{\mset}[2]{ \left\{#1\middle|#2\right\} }

\usepackage[]{xcolor}

\newcommand{\Cinftylim}{C^\infty_{\mathrm{el}}(\R)}
\newcommand{\so}{\mathfrak{so}}

\begin{document} 
\title{On the singularities of the exponential function of a semidirect product}

 \author{Alexandru Chirvasitu\footnote{University at Buffalo, Buffalo, NY, United States of America: \href{mailto:achirvas@buffalo.edu}{achirvas@buffalo.edu}}, Rafael Dahmen\footnote{KIT Karlsruhe, Germany},\\
 Karl--Hermann Neeb\footnote{Department Mathematik, FAU,
 Erlangen, Germany: \href{mailto:neeb@math.fau.de}{neeb@math.fau.de}},
 Alexander Schmeding\footnote{NTNU Trondheim, Norway: \href{mailto:alexander.schmeding@ntnu.no}{alexander.schmeding@ntnu.no}}}
\date{}
{\let\newpage\relax\maketitle}

\begin{abstract}
  We show that the Fr\'echet--Lie groups of the form $C^{\infty}(M)\rtimes \mathbb{R}$ resulting from smooth flows on compact manifolds $M$ fail to be locally exponential in several cases: when at least one non-periodic orbit is locally closed, or when the flow restricts to a linear one on an orbit closure diffeomorphic to a torus. As an application, we prove that the Bondi--Metzner--Sachs group of symmetries of an asymptotically flat spacetime is not locally exponential.
\end{abstract}

\medskip

\\[1em]


\section{Introduction} 
For infinite-dimensional Lie groups beyond the Banach space setting, the Lie group exponential in general does not induce a local diffeomorphism near the unit, as singular points might exist arbitrarily close to the unit. Groups exhibiting this pathology are called \emph{non-locally exponential}. Examples include diffeomorphism groups and the formal diffeomorphisms of $\R^n$, cf. \cite{Neeb06}. In the present note we investigate when the Lie group exponential of a semidirect product exhibits singularities making the resulting Lie groups non-locally exponential.

Concretely, we are interested in semidirect products arising as the lift of a right Lie group action $\sigma \colon M\times G \rightarrow M$ on a compact manifold $M$ to the space of smooth functions $C^\infty (M)$. These give rise to our main examples of semidirect products
\begin{equation*}
  C^\infty(M) \rtimes_{\alpha} G, \qquad \text{where } \alpha \colon G\times C^\infty (M)\rightarrow C^\infty (M),\quad \alpha_g(F)=F\circ \sigma_g.
\end{equation*}
Note that to study singularities in these semidirect products it suffices to study singularities of the exponential function by restricting to the action of $1$-parameter subgroups of $G$. Every $1$-parameter subgroup of $G$ gives rise to a subgroup $C^\infty(M) \rtimes \R$ of $C^\infty (M)\rtimes G$. Then singularities of the exponential function of the subgroup are inherited by the exponential function of the semidirect product by naturality.

We can and will therefore focus our investigation on semidirect products involving $1$-parameter groups.  Several sufficient criteria for the existence of singularities arbitrarily close to the unit in these semidirect products are established. These allow us to treat important classes of examples such as the Bondi-Metzner-Sachs (BMS) group for asymptotically flat spacetimes from general relativity, \cite{PaS22}. An asymptotically flat spacetime models an isolated gravitational system which approximate Minkowski spacetime "at infinity". These spacetimes admit a unique conformal compactification. While the BMS-group was originally conceived via the construction of a Lie algebra of vector fields satisfying certain decay conditions towards infinity, the compactification of spacetime allows to identify them with certain diffeomorphisms of a regular part at the boundary. In particular, this identifies the BMS group as a semidirect product of the Lorentz group and an infinite dimensional vector space of ``supertranslations". Hence our techniques allow us to prove the following 

\begin{thmN}
  The BMS-group is not locally exponential.
\end{thmN}


This settles the non-exponentiality conjecture on the BMS groups from \cite[Conjecture 3.15]{PaS22} in the affirmative. A more detailed discussion of the relevance of the non-local exponentiality of the BMS group and our other results will be given in the section at the end of this introduction below. While the results for semidirect products of Lie group actions on spaces of smooth functions are new and yield interesting results, we were not able to completely characterize non-local exponentiality of the resulting groups. A summary and aggregate of \Cref{th:3.10} and \Cref{cor:surj.on.s1} reads as follows.


\begin{thmN}
  Let $(M,\sigma)$ be a smooth flow on a compact manifold. The Lie group $C^\infty(M) \rtimes_{\alpha} \R$ is not locally exponential if any of the following conditions holds:
  \begin{enumerate}[(1),wide]
  \item The flow has at least one non-periodic, locally closed orbit. 
  \item The flow restricts to a linear non-trivial one on a torus embedded in $M$.

  \item $M$ admits an equivariant submersion onto a fixed-point-free flow on $\bS^1$.

  \item $M$ admits an equivariant submersion onto a linear flow on a torus. 
  \end{enumerate}
\end{thmN}

Thus we have reason to conjecture the following general statement.

\begin{conjN}\label{cjn:main}
  For every smooth non-trivial flow $\sigma \colon \R \times M\to M$ on a compact manifold $M$, the Lie group $C^\infty(M,\R) \rtimes_\alpha \R$, specified by $\alpha_t(f) \coloneq f \circ \sigma_t$, is not locally exponential.  More specifically, there is a sequence $t_n \in \R_+$ with $t_n \to 0$ such that the points $(0,t_n)$ are singular for the exponential function.
\end{conjN}

Finally, we collect in Section 5 several results on the related case of semidirect products involving a Lie group acting on Banach spaces of functions of finite order differentiability. Also in this case the exponential map fails to be a local diffeomorphism. The physical relevance of these results is discussed in the next section.

\subsection*{Interpretation and consequences of our results for physics}

The BMS group has been proposed as a candidate for ``the symmetry group" of asymptotically flat spacetime. The original derivation of this group in \cite{BBM62,Sach62} was achieved via a gauge fixing approach. Later it was discovered that the BMS group can also be viewed as a group of conformal diffeomorphisms of a certain part of the boundary of the conformal compactification of spacetime. 
In \cite{PaS22} the basic infinite-dimensional Lie theory for the BMS group and some related groups was described. There it was uncovered that the BMS and related groups are not analytic Lie groups. This, together with the pathology uncovered in Theorem A of the present paper, shows that the interplay between the BMS group and its Lie algebra is more intricate than usually in finite dimensions. In particular, there is no sensible way to construct exponential coordinates for the BMS group. 

Since the BMS group is in a certain sense a minimal candidate for a symmetry group of asymptotically flat spacetimes, the complications in the correspondence of Lie algebra and associated Lie group should also be expected for the other groups proposed as candidates for asymptotic symmetry groups.
More specifically, several extensions, such as the Newman--Unti (NU) group, the generalised BMS group (gBMS), the extended BMS group (eBMS) (see e.g. \cite{PaS22} for a summary) and the conformal BMS group  (cBMS, \cite{cBMS17}) have been proposed (arising e.g.~through different choices in the gauge fixing procedure). 
The NU, gBMS and eBMS groups are constructed as pairs of global group objects and associated Lie algebras. These are also semidirect products. Beyond differences in gauge fixing, the symmetry groups discussed so far consider only the (future) null infinity part of the boundary and excise the corner parts of the conformal compactification (cf.~e.g.~\cite{AaR92}). Generalizations of the BMS group taking these parts of the compactification into account, can be expected to exhibit a semidirect product structure, whence we expect them not to be locally exponential.

Another problem is exemplified by the cBMS group. In \cite{cBMS17}, it is described via the generators of an infinite-dimensional Lie algebra. For this example and related proposed generalizations of the BMS group, our results imply that the passage to a global object will require a serious effort, as a  ``simple attempt to exponentiate the Lie algebra" will most probably not work. Refined tools, such as criteria as to when extensions of infinite-dimensional Lie algebras integrate to Lie group extensions might help to resolve the problem in the passage from infinitesimal to global level for asymptotic symmetry groups \cite{ne02, ne04, ne07}. 

Finally, we would like to point out that the results in Section \ref{se:extra} are not only of use to derive Theorem B, but also have significance for the physical theory. It has been discussed in the literature whether the BMS group should not be a semidirect product of the Poincar\'{e} group with a space $C^k(M)$ of $k$-times continuously differentiable functions for $2\leq k < \infty$. While \cite{PaS22} shows that in this case no Lie group structure is obtained, one might still wonder whether the topological group obtained via the construction still exhibits good mathematical properties useful for the physical theory, for instance 
as a half Lie group as in \cite{mn18}. As our results in Section \ref{se:extra} show, there is no hope to recover a well behaved analogue of the Lie group exponential in this setting. Hence we conclude that these groups can not be used to repair the pathology of the BMS group uncovered in Theorem A of the present paper.

\textbf{Acknowledgements} KHN acknowledges support by DFG-grant NE 413/10-2.


\section{Groups acting on vector spaces}\label{sect:GE}

In this section we consider first a general setup for a simple class of semidirect products arising from a group action on a locally convex vector space. Here and in the following we write \textit{locally convex space} as a shorthand for Hausdorff \textit{locally convex topological vector space}.

\begin{defn} 
  Let $G$ be a \textit{Lie group} \cite[Definition II.1.1]{Neeb06}, $E$ a locally convex space, and $\alpha \colon G \rightarrow \GL(E), g \mapsto \alpha_g$ a homomorphism for which the corresponding action map
  \begin{equation*}
    \alpha \colon G \times E \to E,\quad (g,v) \mapsto \alpha_g(v)
  \end{equation*}
  (denoted abusively by the same symbol) is smooth in the Bastiani sense. Here a map $f \colon E \supseteq U \rightarrow F$ on an open subset $U$ of a locally convex space $E$ is (Bastiani-)smooth if for every $x \in U, v_1, \ldots v_k \in E,  k\in \N_0$ the iterated $k$-fold directional derivatives $d^kf(x)(v_1,\ldots ,v_k)$ exist and yield continuous mappings $d^kf \colon U \times E^k \rightarrow F$, cf. \cite{Schm23}. The chain rule holds for Bastiani-smooth maps, hence manifolds, smooth functions on them and Lie groups can be defined in the usual way using charts.

  Interpreting the vector space $E$ as an abelian Lie group, we construct the semidirect product Lie group $E \rtimes_\alpha G$.  Its Lie algebra is the semidirect product algebra \break $\Lf(E \rtimes_\alpha G)=E \rtimes_{\Lf(\alpha)}\Lf(G)$ induced by the derived action, given for $x \in \Lf(G)$ by
  \begin{equation*} \Lf(\alpha)(x)v = T_e(\alpha^v)x \quad \mbox{ for } \quad
    \alpha^v(g) = \alpha_g(v).\end{equation*}
\end{defn}
For a Lie group $G$ we call a smooth function $\exp \colon \Lf(G) \to G$ an \textit{exponential function} if all curves $\gamma_x \colon \R \to G, \gamma_x(t) := \exp(tx)$, are one-parameter groups and $\gamma_x'(0) = x$. The existence of an exponential function is standard for Banach Lie groups, but there exist infinite-dimensional Lie groups without exponential function, see \cite[FP6]{Neeb06}.  We investigate the exponential function of the semidirect product ${E \rtimes_\alpha G}$. Let us recall the general form of the exponential function of such groups (see \cite[Ex.~II.5.9]{Neeb06}). 

Recall \cite[post Def.~5.2]{trev_tvs} that \textit{sequentially complete} topological vector spaces are those whose Cauchy sequences converge. This ensures the existence of integral operators such as \Cref{eq:beta_t} below (see the discussion following \cite[Def.~I.1.4]{Neeb06}). We assume locally convex spaces carrying actions sequentially complete unless explicitly stating otherwise.

\begin{lem}
  Suppose that $G$ has an exponential function $\exp_G \colon \Lf(G) \to G$, and that $E$ is sequentially complete, so that for $x \in \Lf(G)$ the integrals $\beta_x := \int_0^1 \alpha(\exp_G (sx))\dd s$ exist pointwise and define continuous linear maps on~$E$.  Then the exponential map of $E \rtimes_\alpha G$ is
  \begin{align}\label{exp_semidir}
    \exp(v,x) =\left(\beta_x v, \exp_G (x)\right)
  \end{align}
\end{lem}

This is  verified easily by showing that the curves $t \mapsto \exp(tv,tx)$ are one-parameter groups with tangent vector $(v,x)$ 
in $t = 0$.
In particular, if $G=\R$, we obtain as a special case
the exponential map of $E \rtimes_\alpha \R$:
\begin{align}\label{exp:1para}
  \func{\exp_{E \rtimes_\alpha \R}}{E \rtimes \R}{E \rtimes \R}{(v,t)}{(\beta_tv,t)}\\
  \text{with } \beta_t = \int_0^1 \alpha(st) \dd s = \frac{1}{t} \int_0^t \alpha(s) \dd s. \label{eq:beta_t}
\end{align}
In this case it suffices to study the integral operators $\beta_t$ to deduce whether the group $E \rtimes_\alpha \R$ is locally exponential or not.  We remark that \cite[Ex.~II.5.9 (a)]{Neeb06} gives an easy example in which the group $E \rtimes_\alpha \R$ fails to be locally exponential for an action
on a \textit{\Frechet\, space} $E$ (i.e.~a locally convex
completely metrizable space,  \cite[\S 10]{trev_tvs});
we shall see more below.
We fix some notation.

\begin{defn}
  Let $\alpha \colon \R \times E \rightarrow E, \alpha_t (v) := \alpha^v(t):= \alpha (t,v)$ be a left Lie group action on the locally convex vector space $E$. For the infinitesimal generator of $(\alpha_t)_{t \in \R}$ we write
  \begin{align}\label{inf-generator}
    D \coloneq \frac{d}{dt}\Big|_{t=0} \alpha_t.  
  \end{align}
\end{defn}
In the following we will study the invertibility of the operators
$\beta_t$. In this context, the relation
\begin{equation*}
  \frac{\alpha_t - \one}{t} = D \circ \beta_t 
\end{equation*}
will prove useful. It follows immediately from
\begin{equation}
  \label{eq:alpha-T-rel}
  \frac{\alpha_t - \one}{t}
  = \frac{1}{t} \int_0^t \frac{d}{ds} \alpha_s\, ds 
  = \frac{1}{t} \int_0^t D \alpha_s \, ds 
  = D \circ \beta_t.
\end{equation}

This discussion leads to:

\begin{prop}\label{prop:injec}
  For $T > 0$, the operator $\beta_T$ is not injective if and only if there exist $T$-periodic vectors which are not fixed by $(\alpha_t)_{t\in \bR}$.
\end{prop}
\begin{proof} 
  Suppose first that $\beta_T$ is not injective.  In view of \Cref{eq:alpha-T-rel}, $\alpha_T - \one$ vanishes on $\ker\beta_T$.  Therefore the flow on the closed subspace $\ker \beta_T$ is $T$-periodic, so that we obtain a representation of the compact group $\T_T := \R/T\Z$.  The characters of this group are of the form $\chi_m \colon \T_T \to \T, \chi_m([t]) = e^{2\pi i m t/T}$, $m \in \Z$.  For an action on a sequentially complete real locally convex space $E$, applying the general Peter--Weyl Theorem \cite[Thm.~3.51]{HaM23} to the representation on the complexified space $E_\bC$, shows that $F := \ker \beta_T$ is topologically generated by the eigenspaces
  \begin{equation}
    \label{eq:D-ev}
    F_m := \ker\Big(D^2 + \frac{4\pi^2  m^2}{T^2} \one\Big), \quad m \in \Z \setminus \{0\}.    
  \end{equation}
  Actually, the Peter--Weyl Theorem first  implies
  that the eigenvectors span a dense subspace,
  but the eigenspaces $F_{\bC}(\chi_m)$
  combine to the corresponding real $D$-eigenspace
  \begin{equation*} F_m = F \cap (F_{\bC}(\chi_m) + F_\bC(\chi_{-m})). \end{equation*}
  In view of this relation, the inclusion $F_m \subseteq \ker \beta_T$ follows from the fact that every $\chi_m$-eigenvector 
  is killed by $\beta_T$:
  \[  \beta_T v = \frac{1}{T}\int_0^T \alpha_s v\, ds 
  = \frac{1}{T} \int_0^T e^{2\pi i m s/T} \, ds \cdot v 
  = 0 \cdot v = 0.\]
  
  Further $\beta_T|_{\ker D} = \id$ implies that some
  $F_m$ is non-zero. 
  We thus obtain $T$-periodic  vectors that are not fixed.

  Conversely, if $T$-periodic vector exists which is not fixed by the action, then one of the eigenspaces $E_m, m \not=0,$ as in \Cref{eq:D-ev} is non-trivial by the Peter--Weyl Theorem \cite[Thm.~3.51]{HaM23}.  In view of \Cref{eq:alpha-T-rel} and the injectivity of $D$ on each $E_m$, these eigenspaces are contained in $\ker(\beta_T)$.
\end{proof}

\begin{prop}\label{prop:eigenvalue_seq}
  If $i\lambda \in i \R^\times$, is an eigenvalue of $D$ on $E_\bC$, then $\beta_T$ is not injective for $T \in \frac{2\pi}{\lambda} \Z$. 
\end{prop}

\begin{proof}
  Let $v = v_1 + i v_2 \in E_\bC$ be a corresponding eigenvector. Then
  \begin{equation}\label{eq:betatv}
    \beta_T v
    =
    \int_0^T e^{i\lambda s}v\, ds
    = \frac{1}{i\lambda}(e^{i\lambda T} - 1)v,
  \end{equation}
  which vanishes for $0 \not=T \in \frac{2\pi}{\lambda} \Z$, so that $\beta_T(v_1) = \beta_T(v_2) = 0$.
\end{proof}

Now applying the observation of the previous proposition to $1$-parameter groups obtained from an arbitrary Lie group action we deduce the following.

\begin{cor} \label{cor:2.2} Let $G$ be a Lie group acting from the left on a locally convex vector space $E$. If there exists a $1$-parameter group in $G$ such that the infinitesimal generator $D$ of its action on $E_\bC$ admits an unbounded sequence of purely imaginary eigenvalues, then $E \rtimes_\alpha G$ is not locally exponential.
\end{cor}

\begin{prop}\label{prop:2}
  If there exists a continuous linear eigenfunctional $\nu \colon E_\bC \to \bC$ which is $T$-periodic and not fixed by $(\alpha_t)_t$, then $\beta_T$ is not surjective.
\end{prop}
\begin{proof}
  Indeed, any such eigenfunctional will annihilate the range of $\beta_T$ by \Cref{eq:betatv}: choose $n \in \Z$ such that $\nu \circ \alpha_t = e^{2\pi i n t/T} \nu$ for $t \in \R.$ By assumption, $n \not=0$, so that
  \begin{equation*}
    \nu(\beta_T(v))= \int_0^T e^{2\pi i n t/T}\, dt \cdot \nu(v) = 0
  \end{equation*}
  implies the assertion.
\end{proof}

\begin{rem}\label{re:deriv}
  If $E$ is a complex algebra and $D$ a derivation, then $Dv = i\lambda v$, and thus
  \begin{equation*}
    Dv^n = n i\lambda v^n.
  \end{equation*}
  If all powers $v^n$ are non-zero, then all numbers $n i \lambda$ are eigenvalues of $D$, and therefore $E \rtimes_\alpha G$ is not locally exponential.
\end{rem}

\section{Semidirect products with spaces of smooth functions}\label{se:semidir-smth}

We work with spaces $C^k(M)\coloneq C^k(M,\R)$, $0\le k\le\infty$ of functions on smooth (i.e.~$C^{\infty}$) finite-dimensional manifolds $M$, with or without a boundary (see \Cref{res:bdry}\Cref{item:res:bdry:bdry}). Recall that for a finite-dimensional manifold $N$, the compact-open topology on $C^0(N,\R^n)$ is the locally convex topology induced by the seminorms $\lVert f\rVert_K \coloneq \sup_{x \in K} \lVert f(x)\rVert$, where $K \subseteq N$ is compact. Then the compact-open $C^k$-topology, is the locally convex topology turning the map
\begin{equation}\label{emb:Ck_top}
  C^k (M) \rightarrow \prod_{0 \leq p \leq k} C^0(T^p M, T^p\R)
  \cong \prod_{0 \leq p \leq k} C^0(T^p M, \R^{2^p}),\quad f \mapsto \left(T^p f\right)_{0\leq p \leq k}
\end{equation}
into a topological embedding, where the spaces on the right are endowed with the compact-open topology, cf.\ \cite[5.1]{Neeb14}. 
Localizing with the help of charts as in \cite[Lemma 2.5]{Schm23}, we exploit that, on open subsets of euclidean space, the toplogy coincides with the topology of uniform convergence on compact sets of functions and derivatives up to order $k$. By \cite[Ex.~10.1]{trev_tvs} 
$C^k(M)$ is a Fr\'echet space.

For compact $M$, the space $C^k(M)$ with finite $k$, will be Banach spaces with respect to the compact open $C^k$-topology. In particular, $\|\cdot\|_0$ will denote the supremum norm on $M$. For $k>0$ we note that norms inducing this topology measure the supremum of a function and its derivatives up to order $k$. Write $\|\cdot\|_r$ for a norm inducing the compact open $C^r$-topology on $C^k(M)$, $0 \leq r \leq k$. Adjusting choices, we may assume that 
$$\|f\|_r \leq \|f\|_s \text{ for all } r \leq s, f\in C^k(M).$$   

Moreover, $C^{\infty}(M)$ is in any case (\cite[\S 1.46]{rud_fa} or \cite[Prop.~34.4]{trev_tvs}) a Fr\`echet and \textit{Montel} space (an \textit{(FM)-space} as in \cite[\S 27.2]{k_tvs-1}) in the sense of \cite[Def.~34.2]{trev_tvs}, i.e. it is a \textit{barreled} Fr\`echet space \cite[Def.~33.1]{trev_tvs} such that closed bounded subsets are compact.


Consider a smooth action (or \textit{flow})
\begin{equation*}
  \bR\times M
  \ni
  (t,p)
  \xmapsto{\quad}
  \sigma_t(p) =: \sigma^p(t)
  \in M
\end{equation*}
on a compact smooth manifold $M$, with associated \textit{infinitesimal generator} \cite[Thm.~9.12]{lee2013introduction} $X$ (a smooth vector field on $M$):
\begin{equation}  \label{eq:xalpha}
  X_p = \frac{d}{dt}\bigg|_{t=0}\sigma_t(p)\in T_pM
  ,\quad
  p\in M.
\end{equation}
The orbits $\sigma_{\bR}(p)$ are precisely the \textit{integral curves}
of $X$ \cite[pre Ex.~9.1]{lee2013introduction}.

The action of $\bR$ on $M$ induces one on each $C^k(M)$, $0\le k\le \infty$ defined by
\begin{equation}\label{eq:act.on.ck}
  \bR\ni t
  \xmapsto{\quad}
  \alpha_t\in \Aut(C^k(M))
  ,\quad
  \alpha_tf:=f\circ\sigma_t.
\end{equation}
The material in \Cref{sect:GE} then applies to $E:=C^k(M)$ (we will be mostly interested in $C^{\infty}$) carrying the $\bR$-action $\alpha$. 

\begin{rems}\label{res:bdry}
  \begin{enumerate}[(1),wide]
  \item\label{item:res:bdry:bdry} Whether smooth manifolds $M$ are allowed a boundary $\partial M$ will make little difference to the discussion, so we will not revisit the issue: flows on $M$ will leave the boundary invariant by the latter's topological invariance \cite[Thm.~1.37]{lee2013introduction}, whence one can always complete such a flow to one on the boundary-less \textit{double}
    \cite[Ex.~9.32]{lee2013introduction} $M\cup_{\partial M}M$.

    See also \cite[Thm.~9.34]{lee2013introduction} and surrounding discussion for remarks to the effect that the theory of flows and integral curves holds good in the presence of a boundary so long as one restricts attention to vector fields whose restriction to $\partial M$ is tangent to it.

  \item\label{item:res:bdry:ext2l1} The $\bR$-actions on $C^{k}(M)$ resulting from flows on $M$ extend to the space $L_c^1(\bR)$ of compactly-supported integrable functions with respect to the usual Lebesgue (=Haar) measure in the familiar fashion
    (\cite[Thms.~19.18, 22.3]{hr-1}), by
    \begin{equation}\label{eq:ext2l1}
      \alpha_\varphi f:=\int_{\bR}\varphi(s)\alpha_s f\ \mathrm{d}s
      ,\quad
      \varphi\in L_c^1(\bR)
      ,\quad
      f\in C^{k}(M).
    \end{equation}
    In terms of that extension, we have
    \begin{equation}\label{eq:betas}
      \beta_s = \frac 1s \alpha_{\chi_{[0,s]}}
      ,\quad
      \chi_{A}:=\text{the characteristic function of the subset $A\subset \bR$}. 
    \end{equation}
    The typical setup for \Cref{eq:ext2l1} is that of actions on \textit{Banach} spaces, but the extension does make sense here: for any compact interval $I \subseteq \R$ the partial orbit $\alpha^f(I)$ is compact, hence has compact closed convex hulls \cite[\S 20.6(3)]{k_tvs-1}, so \Cref{eq:ext2l1} falls within the scope of standard vector-valued integration machinery \cite[Thm.~3.27]{rud_fa}.
    
    Moreover, \Cref{eq:ext2l1} makes sense for arbitrary $\varphi\in L^1$ whenever the orbit map $\alpha^f$ is bounded: Simply extend the continuous linear map $C_c(\R) \to E, \varphi \mapsto \alpha_\varphi f$ continuously to $E$.
  \end{enumerate}  
\end{rems}

In preparation for results concerning actions on arbitrary manifolds, we first consider operators on $C^{\infty}(\bR)$. 


\begin{defn}[The period $s$-difference and integral operators]
  For $s \in \R^\times$ we define the operators
  \begin{equation*}
    \Delta_s \colon C^\infty (\R) \rightarrow C^\infty (\R),\quad \Delta_s f (x)\coloneq f(x+s)-f(x).
  \end{equation*}
  and similarly, adapting \Cref{exp:1para} with $\alpha_t(f)(x) = f(x+t)$,
  \begin{equation*}
    \beta_s \colon C^\infty (\R) \rightarrow C^\infty (\R),\quad
    \beta_s f(x)=\int_0^1 f(x+st)\dd t.
  \end{equation*} 
\end{defn}

\begin{lem}\label{lem:s-per}
  The kernel of the operator $\beta_s$ consists of all derivatives of $s$-periodic functions.
\end{lem}
\begin{proof}
  The kernel of $\Delta_s$ is the closed subspace $C^\infty_s(\R)$ of $s$-periodic functions. If $F\in C^\infty (\R)$ and $f=F'$, then the function $\beta_sf$ (with $\beta_s$ as in the previous section) satisfies the relation 
  \begin{equation} \label{eq:beta_delta}\beta_s f (x)
    = \int_0^1f(x+st)\mathrm{d}t
    = \frac{1}{s} \int_0^s F'(x+t)\mathrm{d}t 
    =\frac{1}{s}\Delta_s F(x). \qedhere\end{equation}
\end{proof}

\begin{setup}[The image of both operators contains the compactly supported smooth functions]\label{setup:3.2}
  Consider $g \in C^\infty_c (\R)$, then we exploit the compact support of $g$ to define for $s\not=0$ a smooth function
  \begin{equation*}
    f(x)\coloneq \sum_{k=1}^\infty g(x-ks)=g(x-s)+g(x-2s)+\cdots \in C^\infty (\R)
  \end{equation*}
  and we note that $\Delta_s f =g$. Likewise, the compact support of $g'$ allows us to define a smooth function
  \begin{equation*}
    h(x)\coloneq s\cdot \sum_{k=1}^\infty g'(x-ks)=s(g'(x-s)+g'(x-2s)+\cdots)
  \end{equation*}
  such that $\beta_s h = g$.
\end{setup}

\begin{lem}\label{setup:limits_ex}
  We consider the space
  \begin{equation*}
    \Cinftylim
    \coloneq
    \mset{f\in C^\infty(\R)}{\lim_{x\to+\infty}f(x)\text{ and }\lim_{x\to-\infty}f(x) \text{ exist in }\R}\,.
  \end{equation*}
  The map $\beta_s$ is injective on $\Cinftylim$, but $\beta_s(\Cinftylim)$ does not contain $C_c^{\infty}(\bR)$.
\end{lem}
\begin{proof}
  Injectivity is easy to see: by \Cref{lem:s-per} the kernel of $\beta_s$ consists of $s$-periodic functions, which must thus be constant if they have limits at $\pm\infty$. But such a function is also required by \Cref{lem:s-per} to have a periodic antiderivative, so it must in fact vanish.
  
  To show that
  \begin{equation*}
    \beta_s(\Cinftylim) \cap C^\infty_c(\R) \not= C_c^\infty(\R),
  \end{equation*}
  let $g\in C^\infty_c (\R)$ and let $f\in C^\infty(\R)$ be a smooth function with $\beta_s f = g$. By \Cref{setup:3.2}
  \begin{equation*}
    f_0(x)\coloneq s\sum_{k=1}^\infty g'(x-ks)=s(g'(x-s)+g'(x-2s)+\cdots )
  \end{equation*}
  satisfies $\beta_s f_0 = g$, so that
  \begin{equation*}
    f = f_0 + p, \quad \mbox{ with } \quad p\in\ker\beta_s
  \end{equation*}
  with $p$ the derivative of an $s$-periodic function.
  
  For all $x < \min(\supp g)$,
  we have that $f(x)=p(x)$.  If $f\in \Cinftylim$, then the limit for $x\to-\infty$ exists, this implies that $p=0$ since a periodic function having a limit has to be constant and the only constant which is the derivative of a periodic function is zero.  Now, take the limit for $x\to\infty$ and we see that for $M \coloneq \max(\supp g)$, $f_0|_{[M,\infty)}$ is $s$-periodic. More concretely, the $s$-periodization
  \begin{equation*}
    f_1(x) := s\cdot \per_s(g')(x) \coloneq s\sum_{k \in \Z} g'(x-ks) = s\cdot (\per_s g)'(x)
  \end{equation*}
  coincides with $f_0$ for large $x$.  By the preceding argument, $\beta_s(f_0 - f_1) = g$ implies $f_1 = s\cdot \per_s(g') = 0$.  Note that $\per_s (g') =0$ is not satisfied for all $g \in C_c^\infty (\R)$, proving that $\beta_s(\Cinftylim)$ does not contain $C^\infty_c(\R)$.
\end{proof}

\Cref{setup:limits_ex} permits us to treat certain aperiodic $\R$-actions. We illustrate this for the case where $M$ is a circle.
\begin{ex}[Vector fields on the circle with zeros]\label{ex:stat_circ}
  Let $\SSS^1$ be the unit circle identified with $\R /2\pi\Z$ by $\R/2\pi\Z \to \SSS^1, \theta \mapsto e^{i\theta}$ .  Then $2\pi$-periodic functions in $C^\infty (\R,\R)$ correspond to vector fields on the circle.  Let $0 \not=Z$ be a vector field on $\bS^1$ with a zero~$\theta_0$.  The flow of $Z$ is an $\R$-action $\alpha_Z \coloneq \Fl^Z \colon \R \times \SSS^1 \rightarrow \SSS^1$ on $\SSS^1$.  Choosing some $e^{i\theta_1} \in \SSS^1$ in which $Z$ does not vanish as the initial point of the flow curve, we pull back the functions in $C^\infty (\SSS^1)$ along the resulting flow curve $\gamma(t) := \alpha_Z (t, e^{i\theta_1})$. In both cases, due to the flow approaching the stationary points of the flow, we thus obtain subspaces of $\Cinftylim$ which clearly contain the space $C^\infty_c (\R)$. Thus \Cref{setup:limits_ex} implies that $\beta_s$ is not surjective. We deduce that $C^\infty (\SSS^1) \rtimes_{\alpha_Z} \R$ is not locally exponential.
\end{ex}

\begin{ex} (Translations on $\R$) In this example we connect the integral operator to the discussion in \Cref{sect:GE}. The key is to consider the translation action on the reals. We consider the algebra $\cA = C^\infty_c(\R,\mathbb{C})$ and the flow
  \begin{equation*}
    \alpha_t(f)(x) = f(x + t),\quad  x,t \in \R
  \end{equation*}
  generated by the vector field $X.f = f'$. Note that the integral operator for the exponential in the semidirect product
  $\cA \rtimes_\alpha \R$ constructed from $\alpha$ is simply $\beta_1$ discussed above. The discussion shows that $\beta_1$ is injective (this also follows from \Cref{prop:injec}) but not surjective. Alternatively, one can use \Cref{prop:2} to derive that $\beta_1$ is not surjective: For each $p \in \R$, the functional
  \begin{equation*}
    \lambda_p(f) := \widehat f(p) = \int_\R e^{ipx} f(x)\, dx 
  \end{equation*}
  is continuous and satisfies
  \begin{equation*}
    \lambda_p \circ \alpha_t = e^{-ipt} \lambda_p.
  \end{equation*}
  So it is $1$-periodic for $p \in 2 \pi \Z$. 
\end{ex}

\begin{setup}\label{set:periodic_case}
  Consider now a smooth flow $(\sigma_t)_t$ on a compact manifold $M$, $T$-periodic for some $T>0$ in the sense that $\sigma_T=\id$, with the induced action $\alpha$ of \Cref{eq:act.on.ck} and resulting semidirect product $E \rtimes_{\alpha} \R$.

  The vector field $X$ associated to $\alpha$ as in \Cref{eq:xalpha} then also produces a derivation
  \begin{align}\label{eq:deriv_act}
    D \colon C^\infty (M) \rightarrow C^\infty (M),\quad
    D(f) = \left.\frac{d}{dt}\right|_{t=0} \alpha_t f= Xf\,.
  \end{align}
  on the algebra $C^\infty (M)$ (non-zero, because $X\ne 0$).
        
  We will now apply the results from \Cref{sect:GE} to prove under some conditions that $E \rtimes_{\alpha} \R$ is not locally exponential.  Since $D$ is a derivation, \Cref{re:deriv} implies that it suffices to find a non-zero imaginary eigenvalue of $D$ on~$E_\bC$.  Since $\alpha$ is $T$-periodic, all eigenvalues of $D$ are contained in $\frac{2\pi i}{T}\Z$ and, by the Peter--Weyl Theorem \cite{HaM23}, the eigenvectors span a dense subspace of $E_\bC$.  As $D \not=0$, $D$ has a non-zero eigenvalue $\lambda = 2\pi i \frac{m}{T}$ for an $m\in\Z\setminus\{0\}$. Hence \Cref{prop:eigenvalue_seq} implies that, for
  \begin{equation*}
    t_n := \frac{2\pi i}{n \lambda}= \frac{T}{nm}\neq0 
  \end{equation*}
  the operators $\beta_{t_n}$ are not surjective, whence not invertible. Consequently, $E \rtimes_{\alpha} \R$ is not locally exponential.
\end{setup}

\begin{ex}\label{ex:tori} (Linear flows on tori)  
  Let $M := \T^d, d \geq 2$, $0\neq \theta = (\theta_1,\theta_2, \ldots,\theta_d) \in \R^d$, and consider on $M$ the smooth flow
  \begin{equation*}
    \sigma=(\sigma_t)_t
    ,\quad
    \sigma_t(\bz)
    :=
    (e^{it\theta_1} z_1, \cdots, e^{it\theta_d} z_d).
  \end{equation*}
  For the action on $E = C^\infty(\T^d)$ defined by
  $\alpha_t f := f \circ \sigma_t$, we claim that the operators $\beta_s$ of \Cref{eq:beta_t} will be non-injective as well as non-surjective for a dense set of parameters $s\in \bR_{>0}$. 
  Indeed, the characters $\chi_\bm(\bz) := z_1^{m_1} \cdots z_d^{m_d}$ are $\alpha$-eigenvectors with
  \begin{equation*}
    D \chi_\bm = i(m_1 \theta_1 + \cdots + m_d \theta_d) \chi_\bm.
  \end{equation*} 
  Because $\theta \not=0$, at least one $\theta_j$ is non-zero, and $\bm = m_j \be_j$ yields a non-zero imaginary eigenvalue~$i\lambda$.   As no power of $\chi_\bm$ vanishes, all numbers $\Z i \lambda$ are eigenvalues. But then:
  \begin{itemize}[wide]
  \item $\beta_s$ will fail to be injective for $s\in \frac {2\pi i}{\lambda}\bQ_{>0}$ by \Cref{prop:eigenvalue_seq};

  \item and similarly, it will also fail to be surjective for the same values of $s$ as a consequence of \Cref{prop:2}, upon noting that the flow $(\sigma_t)_t$ preserves the Haar measure on the torus and hence the characters we have identified as eigenvectors can also be cast as eigenfunctionals
    \begin{equation*}
      C^{\infty}(\bT)\ni
      f
      \xmapsto{\quad\chi\quad}
      \int_{\bT} \chi\cdot f\ \mathrm{d}\mu_{\bT} 
      \in \bC.
    \end{equation*}
  \end{itemize}
  In particular, the Fr\'echet--Lie group $E \rtimes_\alpha \R$ is not locally exponential. 
\end{ex}

The following will address a broad class of cases of the introductions \Cref{cjn:main}. 

\begin{thm}\label{th:3.10}
  Let $\sigma=(\sigma_t)_{t\in \bR}$ be a non-trivial smooth flow on a compact manifold~$M$. The operators $\beta_s$ on $C^{\infty}(M)$ are non-surjective in either of the following cases.

  \begin{enumerate}[(1),wide]
  \item\label{item:th:3.10:nonperlcl} For almost all $s$ if there is a non-periodic locally closed orbit $\sigma_{\bR}(p)$. 

  \item\label{item:th:3.10:tor} For a dense set of $s\in \bR_{>0}$ if there exists an $p\in M$, such that $T := \overline{\sigma_\R(p)}$ is a torus on which the flow is linear.
  \end{enumerate}
\end{thm}

\begin{rem}\label{re:almost.per.pt}
  Recall \cite[\S I.3.3, Def.~2, Prop.~5]{bourb_top_en_1} that the \textit{locally closed} subspaces of a topological space are those open in their closures.
  
  That hypothesis of \Cref{th:3.10}\Cref{item:th:3.10:nonperlcl} on $\sigma_\R(p)$ can be recast as asking that the orbit map $\sigma^p$ be non-periodic and the orbit $\sigma_\R(p)$ be \textit{locally connected} under its subspace topology has a basis consisting of connected open subsets (\cite[Def.~27.7]{wil_top}).
\end{rem}

Before going into the proof, we  note that \cite[Ex.~II.5.9]{Neeb06} shows that exponential surjectivity locally at the origin is equivalent to $\beta_s$ being a diffeomorphism for all sufficiently small $s\in \bR_{>0}$. Hence  \Cref{th:3.10} proves the following.

\begin{cor}\label{cor:flows.not.exp}
  For non-trivial flows $(\sigma_t)_{t \in \R}$ as in \Cref{th:3.10} the exponential of the corresponding Fr\'echet--Lie group $C^{\infty}(M)\rtimes_{\alpha}\bR$ is not a diffeomorphism locally at the origin.
\end{cor}

\begin{proof}[Proof of \Cref{{th:3.10}}]
  \begin{enumerate}[(1),wide]
  \item   Fix an integral curve $\iota := \sigma^p$
    as in the statement. Setting 
    \begin{equation*}
      A
      :=
      \iota^* C^{\infty}(M)
      \le C^{\infty}(\bR)
      ,\quad
      \iota^*(f) := f \circ \iota,
    \end{equation*}
    note first that $C^{\infty}_{[u,v]}(\bR)\le A$ (functions supported in $[u,v]$) for some $u<v$. We claim that for almost all $s$ the image $\beta_s(A)$ cannot contain all of $C^{\infty}_c(\bR)$.
    
    
    To see this, observe first that the $[u,v]$-supported $g\in C^{\infty}_{[u,v]}(\bR)$ are precisely the functions of the form $\Delta_s G$ for arbitrary
    \begin{equation*}
      G:=\sum_{k\ge 1}g(\bullet-ks)
    \end{equation*}
    vanishing on $(-\infty,u]$ and $s$-periodic on $[v,\infty)$. By \Cref{eq:beta_delta} we have 
    \begin{equation*}
      \Delta_s(g) = s\beta_s (g'),\quad  g \in C^\infty(\R),
    \end{equation*}
    each $g\in C^{\infty}_{c}(\R)$ is the image through $\beta_s$ of a unique function $\widetilde{g}\in C^{\infty}(\R)$ which
    \begin{itemize}[wide]
    \item vanishes on $(-\infty,u]$

    \item and has $s$-periodic antiderivative on $[v,\infty)$
    \end{itemize}
    The last condition is equivalent to $\widetilde{g}$ being $s$-periodic and having vanishing integral over its length-$s$ intervals in $[v,\infty)$. 
    
    Suppose $g=\beta_s f$ for some $f\in A$. The latter function will be $s$-periodic on both $(-\infty,u]$ and $[v,\infty)$ and hence, for almost all $s$, eventually constant at $\pm\infty$ by \Cref{le:per.at.infty.tame} below (which addresses the restriction $f|_{[v,\infty)}$ directly; the other case is analogous). Because
    \begin{equation*}
      \beta_s\left(f-\widetilde{g}\right)
      =0
      \xRightarrow{\quad}
      f-\widetilde{g}\text{ is globally $s$-periodic}
    \end{equation*}
    and $\widetilde{g}$ vanishes close to $-\infty$, it too must be constant close to both $\pm\infty$ (and indeed, vanish: not only is it periodic, but it has periodic antiderivative at $\pm\infty$). This, though, contradicts the freedom to choose the $s$-periodic-antiderivative $\widetilde{g}|_{[v,\infty)}$ arbitrarily.

  \item The restriction map $C^{\infty}(M)\to C^{\infty}(T)$ is onto by standard smooth-function extension theory \cite[Lemma 2.26]{lee2013introduction}, so the linear-flow non-surjectivity of $\beta_s$ for densely many $s$ (\Cref{ex:tori}) entails non-surjectivity for $C^{\infty}(M)$.  \qedhere
  \end{enumerate} 
\end{proof}

\begin{lem}\label{le:per.at.infty.tame}
  Let $\sigma^p$ be an integral curve meeting the hypothesis of \Cref{th:3.10}\Cref{item:th:3.10:nonperlcl} and fix some $a>0$.  For almost all $s$, the eventually-$s$-periodic continuous functions on $(a,\infty)$ in $(\sigma^p)^*C^\infty(M)$ are eventually constant.
\end{lem}

\begin{proof}
  Let $f \in (\sigma^p)^*C^\infty(M)$.  The assumptions on $f|_{(a,\infty)}$ are its $s$-periodicity and the fact that $f \in (\sigma^p)^*C^\infty(M)$. As $\sigma^p$ takes values in the \textit{compact} manifold $M$, there is a convergent subsequence
  \begin{equation}\label{eq:fconv}
    \sigma^p(n_k)
    \xrightarrow[k]{\quad}
    x\in M
    \xRightarrow{\quad}
    f(n_k)
    \xrightarrow[k]\quad
    f(x)
    ,\quad \left(n_1<n_2<\cdots\right) \subset \bZ_{>0}.
  \end{equation}
  The crucial ingredient, next, is Weyl's result \cite[Thm.~1.4.1]{kn_unif-distrib} to the effect that for almost all $s>0$ the sequence $\left(\frac{n_k}s{}\right)_k$ is \textit{uniformly distributed modulo 1} in the sense of
  \cite[Ch.~1, Def.~1.1]{kn_unif-distrib}:
  \begin{equation*}
    \lim_{n\to\infty}
    \frac{\sharp\left\{1\le k\le n\ :\ x_k\in J\right\}}{n}
    =
    \text{length of }J
    ,\quad
    \quad \text{ for every subinterval }J\subseteq [0,1].
  \end{equation*}  
  This implies in particular that the set of fractional parts $\left\{\frac{n_k}{s}\right\}:=\frac{n_k}{s}-\left\lfloor \frac{n_k}{s}\right\rfloor$ is dense in $[0,1]$ so, rescaling by $s$,
  \begin{equation*}
    n_k = t_k N_k s+(1-t_k)(N_k+1)s
    ,\quad
    N_k\in \bZ_{>0}
    \quad\text{with}\quad
    \left\{t_k\right\}_k\subset[0,1]\text{ dense}.
  \end{equation*}
  The $s$-periodicity of $f$ and the convergence \Cref{eq:fconv} then implies that $f$ is constant (for the almost-all $s$ in question), as claimed.
\end{proof}

\begin{cor}\label{cor:flow.on.s1}
  If $\alpha$ is a non-trivial flow on $M = \bS^1$, then $C^\infty (\SSS^1) \rtimes_\alpha \R$ is not locally exponential.
\end{cor}
\begin{proof}
  Let $X$ be the generating vector field of $\alpha$.  If $X$ has no zeros, then $\alpha$ is periodic, hence a linear flow on the circle with respect to suitable coordinates.  Therefore \Cref{th:3.10}\Cref{item:th:3.10:tor} applies.  If $X$ has a zero, then \Cref{th:3.10}\Cref{item:th:3.10:nonperlcl} applies to every non-fixed point $m \in \bS^1$.
\end{proof}

The criteria above imply in particular that periodic flows on compact manifolds lead to semidirect product groups which are not locally exponential.  Unfortunately, we were not able to prove an analogous criterion for the non-periodic case, but we find it very likely that the groups $C^\infty(M) \rtimes_{\alpha} \R$ are \textbf{never} locally exponential. \Cref{th:3.10} provides sufficient criteria that can be applied in concrete cases.  We demonstrate this in the following example relevant to physics.

\section{Conformal actions on the sphere}

In this section we first explore the actions of the conformal group on the unit sphere.  Our presentation here follows \cite[\S 5]{NaO14}.  As a special case, we then explain how our results solve the non-local exponentiality conjecture from \cite[Conj.~3.15]{PaS22} for the Bondi-Metzner-Sachs (BMS) group.

\begin{setup}\label{conformal_action}
  For $d \in \N$ we let $\SSS^d$ be the unit-sphere in $\R^{d+1}$,
  realized as the subset
  \begin{equation*}
    S := \{ (1,x) \in \R^{1,d+1} \colon \|x\| = 1 \}
  \end{equation*} in $(d+2)$-dimensional Minkowski space. The isotropic (lightlike) vectors in $\R^{1,d+1}$ with positive first component are of the form $\tilde x := \pmat{\|x\| \\ x}$ and the (unit component of the) Lorentz group $G := \SO_{1,d+1}(\R)_0$ acts in the obvious fashion on this set by
  \begin{equation*}
    \pmat{a & b \\ c & m}.\pmat{\|x\| \\ x} = \pmat{ a \|x\| + b x \\ c \|x\| + m x}.
  \end{equation*}
  So $g * x := c \|x\|+ mx$ defines an action of $G$ on $\R^{d+1} \setminus\{0\}$ satisfying $g * (\lambda x) = \lambda (g*x)$ for $\lambda > 0$.  The action on the unit sphere $\bS^d \subseteq \R^{1,d+1}$ is then given by
  \begin{equation*}
    g.x := \frac{g * x}{\|g * x\|}.
  \end{equation*}
  For unit vectors we have $a + b x = \|c + m x\| = \|g * x\|,$ which leads to the cocycle relation
  \begin{equation*}
    \|(g_1 g_2) * x\| = \|g_1 * (g_2 * x)\| = \|g_2 * x\| \cdot \|g_1 * (g_2.x)\|.
  \end{equation*}
  Therefore the \emph{conformal factors} $J_g(x) := (1 + bx)^{-1}$ on $\bS^d$ satisfy the cocycle relation
  \begin{equation*} J_{g_1 g_2}(x) = J_{g_1}(g_2.x)  J_{g_2}(x),\end{equation*}
  so that 
  \begin{equation}
    \label{eq:sigma-sphere}
    (\sigma_{g^{-1}} f)(x) := J_g(x) f(g.x)
  \end{equation}
  defines a smooth left action of $G$ on $C^\infty(\bS^d)$.
  Note that the action on $\bS^d$ can also be written as 
  \begin{align}\label{conf:act}
    g.x
    = (a+ b \cdot  x)^{-1}(c+m\cdot x)\quad\mbox{ for } \quad g=
    \begin{bmatrix} a & b \\ c & m\end{bmatrix},
    b \in M_{1,d}(\R), m \in M_d(\R) 
  \end{align}
  (cf.\ \cite[Eq. (27)]{NaO14}).
  This action is faithful because $-\textbf{1}\not\in G$, which follows from 
  \begin{equation*} G \subseteq \OO^+_{1,d}(\R) = \left\{g=
      \begin{bmatrix} a & e \\ c& d 
      \end{bmatrix} \in \OO_{1,d+1}(\R)\middle| 
      a > 0\right\}.\end{equation*}
\end{setup}

\begin{lem} \label{lem:limits} 
  Let $V$ be a finite-dimensional real
  normed space, $S \subseteq V$ its unit sphere 
  and $A \in {\End}(V)$ with real spectrum. Then, for every $v \in S$, 
  the limit of $e^{tA}.v = \frac{ e^{tA}v}{\|e^{tA}v\|}$  for 
  $t \to \infty$ exists and is a fixed point of $e^{tA}$ for $t \in \R$,
  acting on $S$.
\end{lem}

\begin{proof} By  \Cref{thm_realJordan} we have the Jordan decomposition $A = A_s + A_n$ 
  of $A$, so that $A_s$ is diagonalizable and $A_n$ is nilpotent. 
  Write $v = v_0 + v_1$, where $v_0$ is an $A_s$-eigenvector
  for the largest 
  eigenvalue $\lambda$ occurring in $v$, and $v_1$ is a sum of eigenvectors 
  for smaller eigenvalues. Further, let $d \in \N_0$ be maximal with 
  $A_n^d v_0 \not=0$. Then 
  \begin{equation*}
    e^{tA}v = e^{t\lambda} e^{t A_n} v_0 + e^{tA} v_1 
  \end{equation*}
  
  with $e^{-t\lambda} e^{tA}v_1 \to 0$ for $t \to \infty$. 
  Moreover,  $e^{tA_n} v_0$ is a polynomial of degree $d$.
  Therefore 
  \begin{equation*}
    t^{-d} e^{-t\lambda} e^{tA} v \to \frac{1}{d!} A_n^d v_0,
  \end{equation*}
  and thus $e^{tA}.v$ converges to $\frac{A_n^d v_0}{\|A_n^d v_0\|}$ in $S$. 
\end{proof}

\begin{prop}\label{prop:conf_non_exp} Consider the conformal action of
  $\SO_{1,d+1}(\R)$ on 
  the sphere $\SSS^d$ and $0 \not= x \in \so_{1,d+1}(\R)$.
  Then, for $\gamma_x(t) \coloneq \sigma_{\exp(tx)}$,
  the group $C^\infty(\SSS^d,\R) \rtimes_{\gamma_x} \R$ 
  is not locally exponential and, in particular, 
  $C^\infty(\SSS^d,\R) \rtimes {\SO}_{1,d}(\R)$
  is not locally exponential.    
\end{prop}

\begin{proof}
  Recall, e.g. from \cite[I.8]{Kna02} that $\so_{1,d+1}(\R)$ is a semisimple Lie algebra. We can thus apply \Cref{thm:Jordan_semisimple} to obtain a real Jordan decomposition $x = x_e + x_h + x_n$ of $x \in \so_{1,n}(\R)$, where $x_e$ is elliptic, $x_h$ is hyperbolic 
  and $x_n$ is nilpotent and these summands commute pairwise.

  \textbf{Case 1:} If $x_e = 0$, then \Cref{lem:limits} implies that flow lines on the sphere have limits for $t \to \pm \infty$, so that \Cref{th:3.10}\Cref{item:th:3.10:nonperlcl} applies.
  
  \textbf{Case 2:} Now we assume $x_e \not=0$. Using the preceding lemma, we see that the flow generated by $x_h + x_n$ has fixed points, and the set $F$ of these fixed points is invariant under the flow generated by $x_e$.  As $x \not=0$, there exists $m \in F$ not fixed by $x_e$.  Then the orbit closure $N := \overline{\exp(\R x_e).m}$ is a torus on which $x_e$ generates a linear flow with a dense orbit.  We multiply the invariant measure $\mu_0$ on $N$ with a non-constant eigenfunction~$f$ for the $x_e$-action on $N$, say $x_e f= i \lambda f$. Then the measure $\mu := f \mu_0$ on the torus $N$ is an eigenvector for the action of $\exp(\R x)$ in the space of measures, corresponding to a non-zero imaginary eigenvalue $i\lambda$ of $x$. This implies the existence of arbitrarily small $t$ with $\beta_t$ not surjective.
\end{proof}

\subsection*{Non-local exponentiality of the    BMS group}

The Bondi-Metzner-Sachs (BMS) group models symmetries of an asymptotically flat spacetime in general relativity (see e.g. \cite{BBM62,Sach62}). From a Lie theoretic point of view this group and related examples have been discussed in \cite{PaS22}. It was conjectured that the BMS-group is not locally exponential. Recall first the definition of the BMS-group.

\begin{defn}
  For $d \in \N, d\geq 2$, the $d$-dimensional Bondi-Metzner-Sachs (BMS$_d$) group is the Lie group
  \begin{equation*}
    \BMS_d \coloneq C^\infty (\SSS^d) \rtimes_\sigma \SO_{1,d+1}(\R),
  \end{equation*}
  where the right action $\sigma$ is given by \Cref{eq:sigma-sphere}. For $d=2$, the resulting group is simply known as the Bondi-Metzner-Sachs (BMS) group.
\end{defn}
The presentation of the BMS group here differs somewhat from the one in \cite{PaS22}. However, see e.g. \cite[\S 1.2]{T01} for a discussion connecting it with the M\"obius transformation picture used in \cite{PaS22}.
Due to our previous results, the proof of the following result
is now easy. 

\begin{thm}
  The $\BMS_d$-group is not locally exponential. In particular, the $\BMS$-group is not locally exponential. 
\end{thm}

\begin{proof} Let $0\not= x \in \so_d(\R)\subseteq \so_{1,d+1}(\R)$.  Then the linear action of $\exp(\R x) \subseteq \SO_d(\R)$ leaves the sphere $\bS^d \subseteq \R^{1,d+1}$ invariant, so that $J_{\exp tx} = 1$ for all $t \in \R$.  Therefore the subgroup $C^\infty(\bS^d) \rtimes_\sigma \exp(\R x)$ of the $\BMS_d$ group is not locally exponential by \Cref{prop:conf_non_exp} and this implies that the $\BMS_d$ group is not locally exponential. Specializing to the case $d=2$ yields the non local exponentiality of the BMS group.
\end{proof}

\section{Complements: $C^k$ spaces, periodic flows, and topological embeddings}\label{se:extra}

The present section contains a number of results complementing those of \Cref{se:semidir-smth}. In first instance, note that the operators $\beta_s$ of \Cref{eq:betas} are much more easily proven non-isomorphic on $C^k$ spaces for \textit{finite} $k$. To help justify the claim, we record the following simple remark on representations of locally compact abelian groups on Banach spaces. The essence of the discussion is that Banach norms afford the full force of the theory of \textit{Arveson spectra} introduced in \cite{zbMATH03464313} (see also \cite[\S XI.1]{tak2} or \cite[\S 2]{MR679706}).

\begin{prop}\label{pr:lca.act}
  Let $\bG$ be a locally compact abelian group and $\bG\xrightarrow{\alpha}\GL(V)$ a representation on a Banach space $V$ with continuous orbit maps and $\sup_{g \in G}\|\alpha_g\|<\infty$.

  If $\alpha$ is not norm-continuous, then none of the bounded operators
  \begin{equation*}
    \alpha_f
    :=
    \int_{\bG}f(g)\alpha_g\ \mathrm{d}\mu_{\bG}(g)
    ,\quad f\in L^1(\bG), 
  \end{equation*}
  on $V$ are invertible. 
\end{prop}
\begin{proof}
  By \cite[Thm.~2.13]{MR679706} (also \cite[Cor.~XI.1.16]{tak2}) the \textit{Arveson spectrum} \cite[Def.~XI.1.2]{tak2} $\sigma(\alpha)$ of the representation is non-compact; the spectrum of any operator $\alpha_f$, $f\in L^1(\bG)$ on $V$ thus contains $0$ by \cite[eq.~2.7, p.217]{MR679706}, so cannot be invertible.
\end{proof}

\begin{cor}\label{cor:act.on.ck}
  For a non-trivial flow $(\sigma_t)_{t\in \bR}$ on a compact smooth manifold $M$ inducing a bounded family $(\alpha_t)_{t \in \R}$ on $C^k$ for some $0\le k<\infty$, the operators $\alpha_f$, $f\in L^1(\bR)$ are never automorphisms on $C^k(M)$ for $0\le k<\infty$.

  In particular, the operators 
  $\beta_s=\frac{1}{s}\alpha_{\chi_{[0,s]}}$, $s>0$ of \Cref{eq:betas} are never invertible. 
\end{cor}
\begin{proof}
  This will be an immediate consequence of \Cref{th:3.10}, once we confirm that result's hypothesis: the map
  \begin{equation*}
    \bR\ni t
    \xmapsto{\quad}
    \alpha_t
    \in \cL(C^k(M))
    :=
    \left\{\text{bounded operators on $C^k(M)$}\right\}
  \end{equation*}
  is not norm-continuous for $0\le k<0$.  Indeed, norm continuity is equivalent to the $C^k$-norm continuity of the infinitesimal generator $Xf=\frac{d}{dt}\big|_0f \circ \sigma_t$ of $(\alpha_t)_t$ (\cite[Thm.~13.36]{rud_fa}). That continuity fails is easily seen by restricting to a small portion of an integral curve in a coordinate patch: the differentiation operator is not continuous on $C^{\infty}(\bR)$ for the $C^k$-norm, for small functions can have large derivatives.
\end{proof}

Returning to the setup of \Cref{th:3.10}, recall that, by the \textit{Open Mapping theorem} \cite[Thm.~17.1]{trev_tvs}, a morphism $E\xrightarrow{f} F$ of Fr\'echet spaces is onto precisely when it induces a topological and linear isomorphism $E/\ker f\cong F$. The appropriate dual to surjectivity, for a linear map $E\xrightarrow{f}F$, is that it be a \emph{topological embedding}: a homeomorphism onto its image. 

\begin{rem}\label{re:strng.inj.why.dual}
  Equivalently, a morphism $E\xrightarrow{f}F$ of locally convex topological vector spaces is a topological embedding precisely when for every continuous seminorm $p$ on $E$ there is some continuous seminorm $q$ on $F$ with
  \begin{equation*}
    q(f(v))\ge p(v)
    ,\quad
    \forall v\in E. 
  \end{equation*}
  For morphisms of Fr\'echet spaces this is easily seen to also be equivalent to injectivity and having a closed range:

  If the range $f(F_1)$ is closed then the map is an isomorphism of its domain onto its image by the Open Mapping theorem (applicable because the range, being closed in a Fr\'echet space, is itself Fr\'echet).

Conversely, the topological embedding condition implies the continuity of the inverse $f^{-1}\colon f(F_1) \rightarrow F_1$, so the range must be complete (because the domain was) and hence closed in $F_2$.
\end{rem}

A first general remark concerning topological embeddings:

\begin{prop}\label{pr:factor.strng.inj}
  Let $M\xrightarrowdbl{\pi}N$ be a smooth submersion, equivariant for two flows $\sigma^M$ and $\sigma^N$ 
  on $M$ and $N$ respectively.  If $f\in L^1(\bR)$ and $\alpha^N_f$ is not a topological embedding on $C^{\infty}(N)$, then neither is $\alpha^M_f$ on $C^{\infty}(M)$.
\end{prop}
\begin{proof}
  If $\alpha_f$ were a topological embedding on $C^{\infty}(M)$, then its restriction to the closed subspace $C^{\infty}(N)\cong \pi^*C^{\infty}(N) \subseteq C^{\infty}(M)$ would also have this property.
\end{proof}

Nothing like \Cref{th:3.10}\Cref{item:th:3.10:nonperlcl}
can hold in that form assuming only the existence of \textit{periodic} orbits. Noting that such an orbit carries an action of the quotient $\bR\xrightarrowdbl{}\bS^1\cong \bR/\bZ$, we will focus, for the purpose of elucidating matters, on circle actions. \Cref{th:circ.inv.liouv} says effectively that, by contrast to \Cref{th:3.10}\Cref{item:th:3.10:nonperlcl},
\textit{most} operators $\alpha_{\chi_I}$ associated to smooth circle flows $(M,\sigma)$ and segments $I\subset \bS^1$ are invertible.

To make sense of ``most'', recall first (\cite[p.2]{bak_tnt_2022}, \cite[Def.~E.6]{bug_distr-1}, \cite[pre Thm.~2.3]{oxt_meas-cat_2e_1980}) that a \textit{Liouville number} $x\in \bR\setminus \bQ$ is one such that
\begin{equation*}
  \left(
    \forall N\in \bZ_{>0}
  \right)
  \left(
    \exists p\in \bZ,\ q\in \bZ_{>1}
  \right)
  \quad:\quad
  \left|x-\frac pq\right|<\frac 1{q^N}.
\end{equation*}

The following alternative characterization of (non-)Liouville numbers will be more directly useful below; the statement employs the familiar \textit{big O notation} (e.g. \cite[Notation, p.~xxiii]{ten_an-prob-nt_3e_2015})
\begin{equation*}
  f=O(g)
  \quad\text{meaning}\quad
  |f|\le C|g|
  \text{ for some constant }C>0.
\end{equation*}

\begin{lem}\label{le:liouv.poly}
  The irrational non-Liouville numbers are precisely those real numbers $x$ for which
  \begin{equation}\label{eq:liouv.poly}
    \frac
    {1}
    {\left|e^{2\pi i kx}-1\right|}
    =
    O(k^N)
    \text{ for some }N\in \bZ_{>0}.
  \end{equation}
\end{lem}
\begin{proof}
  Irrationality is of course equivalent to the left-hand side of \Cref{eq:liouv.poly} being finite, so we henceforth take it for granted. As to the balance of the claim, observe first that
  \begin{equation*}
    |e^{2\pi i kx}-1|
    =
    2|\sin \pi kx|
    \asymp
    \min\left(\{kx\},\ 1-\{kx\}\right)
    ,
  \end{equation*}
  as a function of $k$, where $\{\cdot\}$ indicates fractional parts and `$\asymp$' means being of the same order (i.e. each function is $O(\text{the other})$; \cite[p.xxiii]{ten_an-prob-nt_3e_2015} again).
  
  That last quantity always dominates $\frac C{k^N}$ for some $C>0$ and positive integer $N$ precisely when
  \begin{equation*}
    \left|kx-p\right|>\frac{C}{k^N}
    \iff
    \left|x-\frac{p}{k}\right|>\frac{C}{k^{N+1}}
    ,\quad \forall p\in \bZ,\ k\in \bZ_{>0}:
  \end{equation*}
  the negation of the defining property of Liouville numbers, in other words. 
\end{proof}

Motivated by \Cref{le:liouv.poly}, we discuss a multi-dimensional version of the Liouville property.

\begin{defn}\label{def:liouv.tuple}
  For $d$-tuples $\bk=(k_j)_{j=1}^d$ and $\bk'=(k'_j)_{j=1}^d$ of real numbers, write
  \begin{equation*}
    \bk\cdot\bk'
    :=\sum_j k_jk'_j
    ,\quad
    |\bk|
    :=
    (\bk\cdot\bk)^{1/2}.
  \end{equation*}
  for $d$-tuples of real numbers. We might occasionally write $\bk^2:=\bk\cdot \bk$.
  
  A tuple $\bth=(\theta_j)_{j=1}^d\in \bR^d$ is
  \begin{enumerate}[(1),wide]
  \item\label{item:def:liouv.tuple:rat.dep} \textit{rationally dependent} if $\bk\cdot \bth\in \bZ$ for some non-zero $\bk \in \Z^d$ (or: the $\theta_j$ and $1\in \bR$ are linearly independent in the $\bQ$-vector space $\bR$);

  \item\label{item:def:liouv.tuple:rat.indep} \textit{rationally independent} otherwise;

  \item\label{item:def:liouv.tuple:liouv} and \textit{Liouville} if rationally independent and
    \begin{equation*}
      \forall N\in \bZ\quad:\quad
      \frac 1{|e^{2\pi i \bk\cdot \bth}-1|}
      \ne
      O(|\bk|^N)
      \quad\text{for}\quad 0\ne \bk\in \bZ^d.
    \end{equation*}
  \end{enumerate}
  For $d=1$ rational (in)dependence specializes back to being (ir)rational respectively and similarly, the Liouville property for tuples specializes (by by \Cref{le:liouv.poly}) to that for numbers.
\end{defn}


\begin{thm}\label{th:lin.flows.precise.inv}
  Let $\sigma=(\sigma_t)_t$ be a linear flow on the torus $\bT^d\cong (\bS^1)^d$ in the sense of \Cref{ex:tori}, given by
  \begin{equation*}
    \sigma_t(z_1,\ \cdots,\ z_d)
    :=
    \left(e^{2\pi i t\theta_1}z_1,\ \cdots,\ e^{2\pi i t\theta_d}z_d\right). 
  \end{equation*}
  The following conditions on an interval $I\subset \bR$ of length $\ell$ are equivalent.
  \begin{enumerate}[(a),wide]
  \item\label{item:th:lin.flows.precise.inv:inv} The operator
$\alpha_{\chi_I}
      :=
      \int_I \alpha_s\ \mathrm{d}s $ is invertible on $C^{\infty}(\bT^d)$.

  \item\label{item:th:lin.flows.precise.inv:inj} $\alpha_{\chi_I}$ is a topological embedding.

  \item\label{item:th:lin.flows.precise.inv:surj} $\alpha_{\chi_I}$ is surjective. 

  \item\label{item:th:lin.flows.precise.inv:nliouv} The tuple $(\ell\theta_j)_{j=1}^d$ is rationally independent and non-Liouville. 
  \end{enumerate}
\end{thm}

\begin{proof}
  The implication \Cref{item:th:lin.flows.precise.inv:inv} $\Rightarrow$ \Cref{item:th:lin.flows.precise.inv:inj} $\&$ \Cref{item:th:lin.flows.precise.inv:surj} is trivial, so the substantial claims are
  \begin{equation}\label{eq:th:lin.flows.precise.inv:two.impl}
    \text{
      \Cref{item:th:lin.flows.precise.inv:inj}
      \text{ or }
      \Cref{item:th:lin.flows.precise.inv:surj}
      $\xRightarrow{\quad}$
      \Cref{item:th:lin.flows.precise.inv:nliouv}
      $\xRightarrow{\quad}$
      \Cref{item:th:lin.flows.precise.inv:inv}
    }.
  \end{equation}
  An application of the \textit{Fourier transform} \cite[\S 1.2]{rud_lc}
  \begin{equation*}
    f\xmapsto{\quad}\widehat{f}
    ,\quad
    \widehat{f}(\bk):=\int_0^1e^{-2\pi i \bk\cdot\mathbf{t}}f(\mathbf{t})
    \ \mathrm{d}\mu_{\T^d}(\mathbf{t})
    ,\quad
    \bk\in \bZ^d
  \end{equation*}
  for the Pontryagin dual pair $(\bT^d,\ \bZ^d)$ will map $C^{\infty}(\bT^d)$ precisely onto the space of (\emph{rapidly decreasing} 
  sequences $(x_{\bk})_{{\bk}\in \bZ}$ for which the norms
  \begin{equation}\label{eq:seq.norms.d}
    \|(x_{\bk})_{\bk}\|_N
    :=
    \Big(
      \sum_{{\bk}\in \bZ^d}
      (1+{\bk}\cdot {\bk})^N\cdot |x_{\bk}|^2
    \Big)^{1/2}
  \end{equation}
  are all finite, topologized as a Fr\'echet space by that family of norms (see e.g. \cite[Ch.~7, Exer.~22]{rud_fa}). The translation action by
  \begin{equation*}
    \bz=e^{2\pi i \mathbf{s}}
    :=
    (e^{2\pi i s_j})_{j=1}^d\in \bT^d
    ,\quad
    \mathbf{s}=(s_i)_{j=1}^d
  \end{equation*}
  transports over to
  \begin{equation*}
    (x_{\bk})_{\bk}
    \xmapsto{\quad {\mathbf{z}}\triangleright\quad}
    \left(\bz^{\bk}x_{\bk}\right)_{\bk}
    =
    \left(e^{2\pi i {\bk}\cdot \mathbf{s}}x_{\bk}\right)_{\bk},
  \end{equation*}
  so that
  \begin{equation}\label{eq:thetal.explicit.d}
    (x_{\bk})_{\bk}
    \xmapsto{\quad \alpha_{\chi_{[0,\ell]}}\quad}
    \left(
      \int_{0}^{\ell}e^{2\pi i s{\bk}\cdot\bth}x_{\bk}\ \mathrm{d}s
    \right)_{\bk}
    =
    \left(
      \frac
      {e^{2\pi i\ell {\bk}\cdot\bth}-1}
      {2\pi i {\bk\cdot\bth}}
      \cdot
      x_{\bk}
    \right)_{\bk}
  \end{equation}
  (the last parenthetic scalar $\frac{e^{2\pi i\ell {\bk}\cdot\bth}-1}{2\pi i {\bk\cdot\bth}}$ being understood to be $\ell$ when ${\bk\cdot\bth}=0$). That operator is plainly injective precisely when $\ell\bth$ is rationally independent, so we need not consider rationally dependent tuples $\ell\bth$ past this point. We can now return to \Cref{eq:th:lin.flows.precise.inv:two.impl}.
  
  \begin{enumerate}[label={},wide]
  \item \textbf{\Cref{item:th:lin.flows.precise.inv:nliouv} $\Rightarrow$ \Cref{item:th:lin.flows.precise.inv:inv}:} Assuming $\ell\bth$ 
    rationally independent non-Liouville, consider the only possible candidate
    \begin{equation}\label{eq:inv2thetal}
      (x_{\bk})_{\bk}
      \xmapsto{\quad \alpha_{\chi_{[0,\ell]}}\quad}
      \left(
        \frac
        {2\pi i {\bk\cdot\bth}}
        {e^{2\pi i\ell {\bk}\cdot\bth}-1}       
        \cdot
        x_{\bk}
      \right)_{\bk}
    \end{equation}
    for an inverse to \Cref{eq:thetal.explicit.d}. The hypothesis ensures that the scalars $\frac{2\pi i {\bk\cdot\bth}}{e^{2\pi i\ell {\bk}\cdot\bth}-1}$ increase at most polynomially in $k$, so this is indeed a continuous operator on the Fr\'echet space defined by the norms \Cref{eq:seq.norms.d}.
    
  \item \textbf{\Cref{item:th:lin.flows.precise.inv:inj} $\Rightarrow$ \Cref{item:th:lin.flows.precise.inv:nliouv}:} To prove the contrapositive, assume this time that $\ell\bth$ \textit{is} Liouville. We can thus find tuples $\bk_{N}$ in $\bZ^d$ with
    \begin{equation}\label{eq:dec.too.fast.d}
      \left|
        \frac
        {e^{2\pi i \ell{\bk}_{N}\cdot\bth}-1}
        {2\pi i {\bk}_{N}\cdot\bth}
      \right|^2
      <
      \frac 1{2^{2{N}}\left(1+{\bk}_{N}^2\right)^{N}}
      ,\quad\forall {N}\in \bZ_{>0}.
    \end{equation}
    Writing $\delta_{\bk_N}$ for the $\bZ^d$-indexed sequence with a single entry of 1 in position ${\bk}_{N}$ and 0s elsewhere, \Cref{eq:seq.norms.d,eq:thetal.explicit.d} imply that
    \begin{equation*}
      \forall {N}\in \bZ_{>0}\quad:\quad
      \|\delta_{\bk_N}\|_0=1
      \quad\text{and}\quad
      \|\alpha_{\chi_I}\delta_{\bk_N}\|_{N}<\frac 1{2^{N}}.
    \end{equation*}
    This, of course, violates the topological embedding requirement for $\alpha_{\chi_I}$ by \Cref{re:strng.inj.why.dual}.
    
  \item \textbf{\Cref{item:th:lin.flows.precise.inv:surj} $\Rightarrow$ \Cref{item:th:lin.flows.precise.inv:nliouv}:} We again address the contrapositive claim, once more having fixed a sequence $(\bk_{N})_{N}$ satisfying \Cref{eq:dec.too.fast.d}. To prove $\alpha_{\chi_I}$ non-surjective, simply note that the $\Z^d$-sequence
    \begin{equation*}
      x_{\bk}:=
      \begin{cases}
        \left|
        \frac
        {e^{2\pi i \ell\bk_{N}\cdot\bth}-1}
        {2\pi i \bk_{N}\cdot\bth}
        \right|^{\frac 12}
        &\text{if $\bk=\bk_{N}$ for some $N$}\\
        0
        &\text{otherwise}
      \end{cases}
    \end{equation*}
    cannot belong to its range.  \qedhere
  \end{enumerate}
\end{proof}

\Cref{th:lin.flows.precise.inv} has some consequences pertaining to periodic flows.

\begin{thm}\label{th:circ.inv.liouv}
  Let $\sigma=(\sigma_t)_{t\in \bR}$ be a non-trivial smooth flow of $\bS^1\cong \bR/\bZ$ on a compact manifold $M$. The following conditions on an arc $I\subseteq \bS^1$ of length $\ell$ are equivalent.

  \begin{enumerate}[(a),wide]
  \item\label{item:th:circ.inv.liouv:inv} The operator
    \begin{equation*}
      \alpha_{\chi_I}
      :=
      \int_I \alpha_s\ ds 
    \end{equation*}
    is invertible on $C^{\infty}(M)$.

  \item\label{item:th:circ.inv.liouv:inj} $\alpha_{\chi_I}$ is a topological embedding.

  \item\label{item:th:circ.inv.liouv:surj} $\alpha_{\chi_I}$ is surjective. 

  \item\label{item:th:circ.inv.liouv:nliouv} The length $\ell$
    is irrational non-Liouville.
  \end{enumerate}
  In particular, the complement of the set of lengths $\ell$ producing invertible $\alpha_{\chi_I}$ has
  Lebesgue measure 0.
\end{thm}

This theorem will be proved after Proposition~\ref{pr:inj.iff.dense}.

\begin{rem}\label{re:fin.vs.inf.k}
  \Cref{th:circ.inv.liouv} contrasts strikingly with \Cref{cor:act.on.ck} (which may as well be placed in the same setup of $\bS^1$-flows, for it applies to those as a particular case): the ``same'' operators are \textit{frequently} invertible on $C^{\infty}$, but \textit{rarely} so on $C^k$ for finite $k$.

  Such phenomena will be familiar from the literature on pathological spectral 
  behavior for operators on non-normable Fr\'echet spaces: see
  \cite[p.~291]{zbMATH03483026}, for instance, for a recollection of an empty-spectrum operator on a space of smooth functions; there too, an abundance of differentiability translates to an abundance of invertibility. 
\end{rem}

In the following, we work with the \emph{spectrum} $\Sp(\alpha)$ of a separately continuous action $\alpha : G\times E \to E$ of a compact abelian group $G$ on a Montel space $E$. For complex $E$ this is
\begin{equation*}
  \begin{aligned}
    \Sp(\alpha)
    &=        \{    \gamma\in \widehat G = \mathrm{Hom}(G,\T)  \ :\   E_{\gamma}\ne\{0\}\}, \quad \mbox{ where } \\
    E_{\gamma}
    &=
      E_{\alpha,\gamma}
      :=\{v\in E\ :(\forall g \in G)
      \ \alpha_g v = \gamma(g)v\}
      \quad\left(\text{the $\gamma$-eigenspace of $\alpha$}\right).
  \end{aligned}    
\end{equation*}
As for \emph{real} Montel spaces, those can always be complexified so as to recover the analogous description of the spectrum. Note also that the material of \Cref{sect:GE} is available, for Montel spaces are automatically sequentially complete (indeed, even \emph{quasi-complete} \cite[Def.~34.1]{trev_tvs}: closed bounded subsets are complete, being compact).

\begin{prop}\label{pr:inj.iff.dense}
  Let $\bS^1\times E\xrightarrow{\alpha} E$ be a continuous action of $\bS^1 = \R/\Z$ on a sequentially complete locally convex space and $I\subseteq \bS^1$ an arc of length $\ell \in (0,1)$.  Then the following assertions hold:
  \begin{enumerate}[(1),wide]
  \item\label{item:pr:inj.iff.dense:inj.rat} The operator $\alpha_{\chi_I}$ on $E$ is non-injective if and only if $\ell$ is an integral multiple of some number $\frac{1}{n}$, $0 \not= n \in \Sp(\alpha)$. In particular $\ell \in \Q$.
    
  \item\label{item:pr:inj.iff.dense:mont.dense.im} 
    $\alpha_{\chi_I}$ is injective if and only if it has dense image.

    
  \item\label{item:pr:inj.iff.dense:mont.onto} If $E$ is Fr\'echet, then $\alpha_{\chi_I}$ is surjective if and only if it is a topological isomorphism.  This is the case if $\alpha_{\chi_I}$ is a topological embedding.
  \end{enumerate}
\end{prop}

\begin{proof}
  There is no harm in once more complexifying and thus working over $\bC$. We will also identify $\bS^1\cong \bR/\bZ$, with $e^{2\pi i t}$ corresponding to (the image of) $t\in \bR$.
  
  \begin{enumerate}[label={},wide]
  \item[\Cref{item:pr:inj.iff.dense:inj.rat}] The kernel of $\alpha_{\chi_I}$ is an invariant subspace, hence non-zero if and only if it contains an $\alpha$-eigenvector.  On an eigenvector $v \in E_\gamma$, $\gamma(t) = e^{2\pi i nt}$, $n \not=0$, we have
    \begin{equation*}
      \alpha_{\chi_I} v = \frac{e^{2\pi i n \ell} - 1}{2\pi i n}v.
    \end{equation*}
    Therefore the kernel of $\alpha_{\chi_I}$ is non-zero if and only if there exists $0 \not= n \in \Spec(\alpha)$ with $n \ell \in \Z$, i.e., $\ell \in \Z \frac{1}{n}$.

  \item[\Cref{item:pr:inj.iff.dense:mont.dense.im}] From \Cref{item:pr:inj.iff.dense:inj.rat} we know that $\alpha_{\chi_I}$ is injective if and only if, for all non-zero $n$ with $E_n \not=\{0\}$, we have $\ell \not\in \Z \frac{1}{n}$. If this is the case, then $\im(\alpha_{\chi_I})$ contains all eigenspaces $E_n$ because on $E_0$ it acts by multiplication with $\ell$.  By the Peter--Weyl Theorem, $\sum_{n \in \Z} E_n$ is dense in $E$.  So that the image of $\alpha_{\chi_I}$ is dense.

      If, conversely, $\alpha_{\chi_I}$ has dense range and $0 \not=n$ with $E_n \not=\{0\}$, then $P_n \circ \alpha_{\chi_I} \not=0$, where $P_n : E \to E_n$ denotes the canonical projection.  As this operator kills all eigenspaces $E_m, m \not=n$, it must be non-zero on $E_n$, on which it multiplies by a multiple of $e^{2\pi i n \ell}-1$. Thus $e^{2\pi i n \ell} \not=1$, so that $\alpha_{\chi_I}$ is injective.

    \item[\Cref{item:pr:inj.iff.dense:mont.onto}] Because bijective morphisms between Fr\'echet spaces are isomorphisms \cite[Theorem 17.1]{trev_tvs}, the first statement follows from the fact that, in view of \Cref{item:pr:inj.iff.dense:mont.dense.im}, surjectivity implies injectivity.  Now suppose that $\alpha_{\chi_I}$ is a topological embedding.  Then it is injective, and by \Cref{item:pr:inj.iff.dense:mont.dense.im} it has dense range. Because the range is also complete, $\alpha_{\chi_I}$ is surjective.  \qedhere
  \end{enumerate}
\end{proof}

\begin{proof}[Proof of \Cref{th:circ.inv.liouv}]
  The second statement follows from the first given that the set of Liouville numbers has measure 0 by \cite[Thm.~2.4]{oxt_meas-cat_2e_1980}.

    \Cref{th:lin.flows.precise.inv} resolves the case $M\cong \bS^1$, so it is enough to prove the following local-to-global principle for $\bS^1$-actions (and a non-trivial arc $I\subseteq \bS^1$):
    \begin{equation}\label{eq:loc2glob}
      \alpha_{\chi_I}\text{ iso on $C^{\infty}(M)$}
      \xLeftrightarrow{\quad}
      \alpha_{\chi_I}\text{ iso on $C^{\infty}(\sigma_{\bR}(p))$}
      ,\quad
      \forall p \in M. 
    \end{equation}

    Bijectivity being equivalent to surjectivity by \Cref{pr:inj.iff.dense}\Cref{item:pr:inj.iff.dense:mont.onto}, the forward implication 
    is immediate: if $\alpha_{\chi_I}$ is surjective on $C^{\infty}(M)$,  it is so on its $\bS^1$-equivariant quotient space
    \begin{equation*}
      C^{\infty}(M)
      \xrightarrowdbl{\quad\text{restriction}\quad}
      C^{\infty}(\sigma_{\bR}(p)).
    \end{equation*}
    The interesting implication is the converse, which we now turn to.

    If the action is trivial there is nothing to do: $\alpha_{\chi_I}$ is the scalar $\ell$ (length of $I$). On the other hand, if the action is non-trivial, then $\ell$ must be irrational non-Liouville (by the $M\cong \bS^1$ case of the theorem)
    and hence $\alpha_{\chi_I}$ is  invertible
    for the standard rotation action on $\bS^1$.

    Write $M$ as a \emph{twisted product} \cite[\S I.6(A)]{bred_cpct-transf}
    \begin{equation}\label{eq:tw.prod}
      M
      \cong
      \bS^1\times_{\bS^1}M
      :=
      \bS^1\times M\big/\big((z,p)\sim (zg,\sigma_{g^{-1}}(p)),\ z,g\in \bS^1,\ p\in M\big).
    \end{equation}
    This is a quotient of $\bS^1\times M$ by the diagonal $\bS^1$-action
    \begin{equation*}
      (z,p)\triangleleft g
      :=
      (zg,\ \sigma_{g^{-1}}(p)),
    \end{equation*}
    so that the quotient map $\bS^1 \times M \to M$ coincides with the action map. That the diagonal action commutes with the obvious left-hand-factor translation action, and the latter induces the original $\sigma$ on $M$ upon making the identification \Cref{eq:tw.prod}. Because $\alpha_{\chi_I}$ is an isomorphism on $C^{\infty}(\bS^1)$ under standard translation, it also operates isomorphically as $\alpha_{\chi_I}\widehat{\otimes}\id$ on
    \begin{equation*}
      C^{\infty}(\bS^1\times M) 
      \cong 
      C^{\infty}(\bS^1)\widehat{\otimes}C^{\infty}(M)       
    \end{equation*}
    (complete tensor product of \emph{nuclear} topological vector spaces \cite[Thm.~51.6]{trev_tvs}). The realization of $M$ as a quotient \Cref{eq:tw.prod} by a free action then gives
    \begin{equation}\label{eq:inv.fns}
      C^{\infty}(M)
      \cong
      \left(
        C^{\infty}(\bS^1)\widehat{\otimes}C^{\infty}(M)
      \right)^{\triangleleft}
      \underset{\text{closed}}{\le}
      C^{\infty}(\bS^1)\widehat{\otimes}C^{\infty}(M),
    \end{equation}
    the superscript denoting invariant functions under the diagonal action $\triangleleft$. Because the diagonal action and translation on the left-hand tensor commute, the isomorphism $\alpha_{\chi_I}\widehat{\otimes}\id$ leaves invariant the closed subspace \Cref{eq:inv.fns} on which it must thus act as a topological embedding (being the restriction of one). The conclusion that
    \begin{equation*}
      \alpha_{\chi_I}|_{C^{\infty}(M)}
      =
      \alpha_{\chi_I}
      \text{ restricted to \Cref{eq:inv.fns}}
    \end{equation*}
    is an isomorphism follows from \Cref{pr:inj.iff.dense}\Cref{item:pr:inj.iff.dense:mont.onto}.  
\end{proof}

Finally, some consequences of the preceding discussion:

\begin{cor}\label{cor:surj.on.s1}
  Let $(M,\sigma)$ be a smooth flow on a compact manifold admitting an equivariant submersion onto either
  \begin{enumerate}[(1),wide]
  \item a fixed-point-free flow on $\bS^1$;

  \item or a non-trivial linear flow on a torus. 
  \end{enumerate}
  Then the operators $\beta_s$ of \Cref{eq:betas} fail to be topological embeddings on $C^{\infty}(M)$ for a dense set of parameters $s\in \bR_{>0}$.
  In particular, the Lie groups $C^{\infty}\rtimes_\alpha \bR$ attached to such flows are not locally exponential.
\end{cor}
\begin{proof}
  This follows from \Cref{pr:factor.strng.inj}, given that the claim holds for periodic flows on $\bS^1$ and linear flows on tori by \Cref{th:lin.flows.precise.inv}.
\end{proof}

\section*{Data availability statement}

No new data were created or analyzed in this study.

\section*{Conflict of interest statement}

All authors declare that there is no conflict of interest or competing interests associated to this work.

\appendix


\section{Real Jordan decomposition in semisimple Lie algebras}

Since it seems to be difficult to cite a sufficiently complete
statement of the real Jordan decomposition in semisimple
real Lie algebras from the literature, we provide a proof in this appendix. 

\begin{thm}\label{thm_realJordan}
  \textrm{(Real Jordan decomposition)}
  Let $V$ be a finite-dimensional real vector space
  and $A \in \End(V)$. Then there exist uniquely determined
  endomorphisms $A_n, A_s, A_h, A_e$ such that
  \begin{itemize}
  \item[(a)] $A_n$ is nilpotent, $A_h$ is diagonalizable, $A_s$ is semisimple,
    $A_e$ is semisimple with purely imaginary spectrum,
  \item[(b)] $A = A_n + A_s = A_n + A_h + A_e$,
  \item[(c)] $A_n, A_h, A_e$ commute pairwise.
  \item[(d)] $B \in \End(V)$ commutes with $A$ if and only if its commutes with
    $A_n$, $A_h$ and $A_e$. 
  \end{itemize}
\end{thm}

\begin{proof} Let $V_\bC$ be the complexification of $V$, 
  $V_\bC^\lambda$ the generalized $\lambda$-eigenspace of $A$ on $V_\bC$, and 
  $V_{\bC,\lambda}$ the $\lambda$-eigenspace of $A$ on $V_\bC$. We put
  \begin{equation*} V_{[\lambda]} := V \cap (V_\bC^\lambda + V_\bC^{\overline\lambda}) \end{equation*}
  and observe that
  \begin{equation*} V = \bigoplus_{\textrm{Im} \lambda \geq 0} V_{[\lambda]}.\end{equation*}
  Let $A = A_s + A_n$ be the Jordan decomposition of $A$ into
  nilpotent and semisimple component. Then $A_s$ and $A_n$ commute
  and $V_\bC^\lambda$ is the $\lambda$-eigenspace of $A_s$ on $V_\bC$.
  We define $A_h$ by $A_h v = \textrm{Re} \lambda \cdot v$ for $v \in V_{[\lambda]}$.
  Then $A_h$ is diagonalizable over $\R$ and commutes with
  $A_n$ and $A_s$, hence it also commutes with $A_e := A_s - A_h$,
  and $A_e$ acts on $V_\bC^\lambda$ by multiplication with
  $i \textrm{Im} \lambda$. This proves (a)-(c). To verify (d), we assume that
  $B$ commutes with $A$. Then $B$ preserves the generalized
  eigenspaces of $A$ on $V_\bC$, hence commutes with $A_s$ and therefore
  also with $A_n$. It also preserves the eigenspaces of 
  $A_h$, hence commutes with $A_h$, and finally also with $A_e
  = A_s - A_h$.
\end{proof}

\begin{thm}\label{thm:Jordan_semisimple}
  Let $L$ be a real semisimple Lie algebra and $x\in L$.
  Then there exist pairwise commuting uniquely determined elements
  $x_n, x_h, x_e \in L$ such that ${x = x_n + x_h + x_e}$ and
  $\ad x = \ad x_n + \ad x_h + \ad x_e$ is the real Jordan decomposition
  of $\ad x$. 
\end{thm}

\begin{proof} Since $\ad \colon L \to \textrm{der}(L)$ is a linear isomorphism,
  it suffices to verify that, for a derivation $D \in \textrm{der} L$,
  the real Jordan components are contained in $\textrm{der} L$.
  For $D_s$ this follows from the fact that
  \begin{equation*} [L_\bC^\lambda, L_\bC^\mu] \subseteq L_\bC^{\lambda + \mu}, \quad
    \lambda,\mu \in \bC.\end{equation*}
  This in turn implies that the real eigenspaces
  \begin{equation*} L_\mu(D_h) = L \cap \Big(\sum_{\textrm{Re} \lambda =  \mu} L_\bC^\lambda\Big)\end{equation*} 
  in $L$ satisfy
  \begin{equation*}
    [L_\mu(D_h), L_\nu(D_h)] \subseteq L_{\mu + \nu}(D_h).
  \end{equation*}
  Therefore $D_h$ is a derivation of $D$. As $D_n = D - D_s$ is a derivation, the same holds for $D_e = D_s - D_h$.
\end{proof}


\addcontentsline{toc}{section}{References}

\def\polhk#1{\setbox0=\hbox{#1}{\ooalign{\hidewidth
  \lower1.5ex\hbox{`}\hidewidth\crcr\unhbox0}}}

\end{document}